\documentclass{amsart}
\usepackage{amsmath,amssymb}
\usepackage[numbers,sort&compress]{natbib}
\usepackage{color}
\usepackage{hyperref}
\numberwithin{equation}{section}
\newtheorem{theorem}{Theorem}[section]
\newtheorem{lemma}{Lemma}
\newtheorem{corollary}[lemma]{Corollary}
\numberwithin{lemma}{section}
\theoremstyle{definition}

\theoremstyle{remark}
\newtheorem{remark}{Remark}
\begin{document}
\date{}
\title{Action minimizing orbits in the 2-center problems with simple  choreography constraint}

 \author{Furong Zhao}
\address{Department of Mathematics and  Physics, Mianyang Normal University, Mianyang, Sichuan, 621000, P.R.China}
\curraddr{}
\email{zhaofurong2006@163.com}
\author{Zhiqiang Wang}
\address{College of Mathematics and Statistics,  Chongqing University, Chongqing 401331, P.R.China}
\curraddr{}
\email{wzq2015@cqu.edu.cn}
\maketitle
\begin{abstract}
The aim of this paper is to study the motion of   $2+n$-body problem where  two  equal
masses are assumed to be fixed. We assume that the value of each fixed  mass is equal to $M>0$ and  the remaining  $n$ moving particles  have equal masses $m>0$.
According to Newton's second law and the universal gravitation law,  the $n$  particles move under the interaction of each other and the affection of the two fixed particles.
Also, this motion has a natural variational structure.  
Under the simple  choreography constraint, we show that the Lagrangian action functional attains its absolute minimum on  a uniform circular motion.



\end{abstract}

\section{Introduction and Main Results }

The N-center problem describes the motion of a test particle in the field of N space fixed
Newtonian centers of attraction. 
The 1-center problem is actually a  Newtonian 2-body problem which has already been solved by Newton.
The 2-center problem was first investigated by Euler in 1760 and  he proved the integrability of such system. 
While for $N\geq 3$, based on the works of Bolotin, Negrini,  Knauf and Klein \cite{6,7,8,19,20}, we know that
the N-center problem is not completely integrable. 
Castelli \cite{9,10}  used the variational approach to study the N-center problem .
In this paper, we also use the variational minimization method to study  a $2+n$-body problem where   two particles with equal masses of value $M>0$ are fixed at positions $C_1=(1,0,0)$ and $C_2=(-1,0,0)$, while the remaining $n(\geq2)$ particles, with their masses equal to $m>0$,  move under the interaction of each other and the affection of  the two fixed particles. 
 Mathematically, we suppose the force of the fixed center and the moving particles generated by a potential of $\frac{1}{r^{\beta}},\frac{1}{r^{\alpha}},\alpha,\beta>0$, respectively.
These equations of the motion for the $n$ moving  particles can be represented as
\begin{align}\label{1.1}
&m\ddot{q}_{i}(t)=\nabla_{q_{i}}V(q_{1},\ldots,q_{n}),~q_{i}\in \mathbb{R}^3,~i=1,\ldots,n,
\end{align}
where the potential function is given by
\begin{align*}
V(q_{1},\ldots,q_{n})=\sum_{1\leq i<j\leq n}^{n}\frac{m^{2}}{|q_i-q_j|^{\alpha}}+\sum_{i=1}^{n}(\frac{mM}{|q_i-C_1|^{\beta}}+\frac{mM}{|q_i-C_2|^{\beta}}),
\end{align*}
and $q(t)=(q_{1}(t),\ldots,q_{n}(t))$ are  orbits of the $n$ moving  particles.
We will analyse the problem from a variational point of view to find the $2\pi$-periodic solutions of the dynamical system (\ref{1.1}). Looking for periodic solutions of (\ref{1.1}) is equivalent to seeking the critical points of  the associated action functional $\mathcal{A}:\mathcal{H}\rightarrow \mathbb{R}\bigcup \{+\infty\}$
\begin{align*}
&\mathcal{A}(q)=\int_0^{2\pi}[\frac{m}{2}\sum_{i=1}^{n}|\dot{q}_{i}(t)|^2+V(q_{1},\ldots,q_{n})]dt
\end{align*}
on the set
\begin{align*}
\mathcal{H}=\{&q(t)\in H_{2\pi}^{1}({\mathbb{R}}, \mathbb{R}^{3n})\mid q_i(t)\neq q_j(t),~\forall i\neq j,~\forall t\in\mathbb{R}; \\
                                                                      & q_k(t)\neq C_1,~q_k(t)\neq C_2,k=1,2,...,n, ~\forall t\in\mathbb{R}\}.
\end{align*}
In order to find  collisionless periodic solution, many methods have been exploited in the last decades(see \cite{1,2,3,4,5,9,10,11,12,13,14,15,16,17,21,22,24,25,27,28,29}).
As well as the usual Newtonian N-body problem, the main difficulties are that in principle critical points  there might be trajectories with collisions and the action functional is not coercive. 
In this paper, we consider that the motion of the $n$ moving particles is a simple choreography in which the bodies lie on the same curve and exchange their mutual positions after a fixed time; i.e.,
$$q_i(t)=q_{i-1}(t+\frac{2\pi}{n}), \ i=2,3,...,n,\  \forall t\in\mathbb{R}.$$
In \cite{4}, Barutello and Terracini  study  the n-body problem with the only constraint to be a simple choreography, but without  fixed centers. They prove that the absolute minimum of the corresponding functional 􏰄 is attained on a relative equilibrium motion  associated with the regular polygon $\tilde{q}(t)$, which is a solution that  the $n$ particles lie at the vertices of a regular $n$-gon and do a uniform circular motion around the center of the regular $n$-gon. Can we get the same result if we consider the $n$ choreographic particles move in our two center problem? We have our main theorem.
\begin{theorem}\label{maintheorem}
For given $\alpha,\beta,m,M>0$, the absolute minimum of $\mathcal{A}$ on $\Lambda$ 􏰄 is attained on a
relative equilibrium motion $\tilde{q}(t)$, 
where 
$$\Lambda=\{q(t)\in\mathcal{H}\mid \int_0^{2\pi} q_i(t)dt=0,q_i(t)=q_1(t+(i-1)\frac{2\pi}{n})~i=1,2,\ldots,n\}.$$
 Moreover we have,  at every instant,
 $\tilde{q}_1(t)$, $\tilde{q}_2(t)$, $\ldots$, $\tilde{q}_n(t)$ lie at the vertices of a regular $n$-gon centered at {origin} in $yoz$-plane and 
 the circumradius of the regular $n$-gon  increases as $m$ increases.
\end{theorem}
In the proof, we will see  that the following two conditions are basically needed:\\
(H1) $\forall q(t)\in \Lambda$, $\int_0^{2\pi} q_i(t)dt=0,~i=1,2,\ldots,n$;\\
(H2) $\forall q(t)\in \Lambda$, $ q_i(t)=q_1(t+(i-1)\frac{2\pi}{n})$, $i=2,3,...,n$.\\
Then applying some inequalities which  will be introduced in Section \ref{section2}, we give the sharp estimates in Section \ref{section3}.
In \cite{4},  Condition (H1) is also used but not emphasized because the functional in the usual $n$ body problem (which is equivalent to our functional $\mathcal{A}$ in the case $M=0$) is invariant under translations.
It is a classic result that the relative equilibrium motion associated with the regular polygon is a trivial solution in the usual $n$-body problem.
However, is  the minimum right the solution of dynamical system (\ref{1.1})?
Now we give a positive answer through the following lemma.
\begin{lemma}\label{palais}(Palais principle of symmetric criticality \cite{23})
Let $G$ be an orthogonal group on a Hilbert space ${H}$. Define the fixed
point space: ${H}_G = \{x \in H | g\cdot x = x, \forall g\in G\}$; if $f \in C^1({H},R)$ and satisfies $f (g\cdot x) = f (x)$
for any $g \in G$ and $x \in{H}$, then the critical point of $f$ restricted on ${H}_G$ is also a critical point of
$f$ on ${H}$.
\end{lemma}
To be precise, we can consider the functional $\mathcal{A}$ on some special symmetric subspace $\mathcal{H}_G\subset\mathcal{H}$ such that $\mathcal{A}|_{\mathcal{H}_G}$ is coercive. Then Lemma \ref{palais} provides that  the critical point on $\mathcal{H}_G$ is also a critical point on the whole space $\mathcal{H}$ if  for all $g\in{G}$ and $q\in\mathcal{H}$, $\mathcal{A}(g\cdot q)= \mathcal{A}(q)$.
\begin{corollary}
Let $\mathcal{G}=<g_{1},g_{2}>$ be a finite group  with the following representations $ \rho:~\mathcal{G}\rightarrow O(3)$, $ \tau:~\mathcal{G}\rightarrow O(2)$ and $ \sigma:~\mathcal{G}\rightarrow \Sigma_n$ 
such that
\begin{equation}\label{groupaction}
g_{j}\cdot q(t)=(\rho(g_{j}) q_{\sigma(g_{j}^{-1})1}(\tau(g_{j}^{-1})t),~\cdots,~ \rho(g_{j}) q_{\sigma(g_{j}^{-1})n}(\tau(g_{j}^{-1})t)), j=1,2,
\end{equation}
where 
\begin{align}
\label{ga1}\rho(g_{1})=id, \sigma(g_{1}^{-1})i=i-1, \tau(g_{1}^{-1})t=t+\frac{2\pi}{n}, \\
\label{ga2}\rho(g_{2})=-id, \sigma(g_{2}^{-1})i=i, \tau(g_{2}^{-1})t=t+\pi.
\end{align}
Then the minimizer of $\mathcal{A}|_{\mathcal{H}_{\mathcal{G}}}$ is  $\tilde{q}(t)$ and thus a uniform circular solution of (\ref{1.1}).
\end{corollary}
\begin{proof}
First we will see the fixed point space $\mathcal{H}_{\mathcal{G}}\subset \Lambda$.
Precisely, the action of $g_{1}$ (\ref{ga1}) is the simple choreography constraint (H2), 
and the coercive condition (H1) can be deduced by  (\ref{ga2}) which is equivalent to the so-called $T/2$-antiperiodic constraint
\begin{equation*}\label{antiperiodic}
q(t)=-q(t+\pi), \forall t \in\mathbb{R}.
\end{equation*}
Then we can see that
$$\mathcal{H}_{\mathcal{G}}=\{q(t)\in\mathcal{H} \mid q(t)=-q(t+\pi),q_i(t)=q_1(t+(i-1)\frac{2\pi}{n}), ~i=1,2,\ldots,n\}
\subset\Lambda.$$
And it is not difficult to see the uniform circular motion 
described in Theorem \ref{maintheorem} $\tilde{q}(t)\in\mathcal{H}_{\mathcal{G}}$.
By Theorem \ref{maintheorem},  $\tilde{q}(t)$ must be the unique minimum of $\mathcal{A}$ on  $\mathcal{H}_{\mathcal{G}}$.
On the other hand, (\ref{groupaction}) is the general way \cite{15} to define the group action ${G}$ for classical  equal masses  N-body problem that makes the functional invariant. 
In our problem, with two fixed centers, we should be careful.
We have $\mathcal{A}(g\cdot q(t))= \mathcal{A}(q(t)), \forall g \in {G}$,  because the two centers are symmetric and have the same mass.
Then Lemma \ref{palais} implies $\tilde{q}(t)$, the minimum of $\mathcal{A}$ on  $\mathcal{H}_{\mathcal{G}}$, is also a critical point of $\mathcal{A}$ on  the whole space $\mathcal{H}$, thus a solution of dynamical system (\ref{1.1}).
\end{proof}
In the end of this section, we reduce our functional $\mathcal{A}$ due to the  choreographic condition (H2),
\begin{align*}
\mathcal{A}(q)&=\int_0^{2\pi}[\frac{m}{2}\sum_{i=1}^{n}|\dot{q}_{i}(t)|^2+V(q_{1},\ldots,q_{n})]dt  \\
&=\int_0^{2\pi}\frac{mn}{2}|\dot{q}_{1}(t)|^2dt+n\int_0^{2\pi}[\frac{mM}{|q_1(t)-C_1|^{\beta}}+\frac{mM}{|q_1(t)-C_2|^{\beta}}]dt\\
&\ \ \ \ +\frac{n}{2}\sum_{j=1}^{n-1}\int_0^{2\pi}\frac{m^{2}}{|q_1(t)-q_1(t+j\frac{2\pi}{n})|^{\alpha}}dt. 
\end{align*}
Dividing by $mn$, we see that seeking the critical points of $\mathcal{A}(q)$ on $\Lambda$ is equivalent to finding the critical points of $\widetilde{\mathcal{A}}(x)$ on $\widetilde{\Lambda}$,
where
\begin{align*}
\widetilde{\mathcal{A}}(x)&=\int_0^{2\pi}[\frac{1}{2}|\dot{x}(t)|^2dt+\frac{M}{|x(t)-C_1|^{\beta}}+\frac{M}{|x(t)-C_2|^{\beta}} 
+\frac{1}{2}\sum_{j=1}^{n-1}\frac{m}{|x(t)-x(t+j\frac{2\pi}{n})|^{\alpha}}]dt,
\end{align*}
\begin{align*}
\widetilde{\Lambda}&=\{x(t)\in H_{2\pi}^{1}({\mathbb{R}}, \mathbb{R}^{3})| \int_0^{2\pi}x(t)dt=0,
x(t)\neq x(t+i\frac{2\pi}{n}),~\forall t\in \mathbb{R},
                     i=1,2,\ldots,n-1 \}.
\end{align*}
In the following, we work on 
$\widetilde{\mathcal{A}}|_{\widetilde{\Lambda}}$ and if $\widetilde{x}(t)$ is the minimizer of $\widetilde{\mathcal{A}}|_{\widetilde{\Lambda}}$, then
 $\tilde{q}(t)=(\tilde{x}(t),~ \tilde{x}(t+\frac{2}{n}\pi),\ldots,~\tilde{x}(t+\frac{n-1}{n}2\pi))$ is a minimum of $\mathcal{A}|_{\Lambda}$.

\section{Some Inequalities }\label{section2}
We recall the famous  Poincar\'{e}-Wirtinger inequality.
Let $x(t)\in H_{2\pi}^{1}(\mathbb{R},\mathbb{R}^{d})$ such that $\int_{0}^{2\pi}x(t)dt=0$, then
$$\int_{0}^{2\pi}|\dot{x}(t)|dt\geq \int_{0}^{2\pi}|x(t)|^{2}dt,$$
where the equality holds if and only if $x(t)=a \cos t+ b\sin t, a, b\in\mathbb{R}^{d}$.
The following lemma   is some kind of generalization. 
\begin{lemma}\label{mainlemma}
Let $x(t)\in H_{2\pi}^{1}(\mathbb{R},\mathbb{R}^{d})$ such that $\int_{0}^{2\pi}x(t)dt=0$, then
$$\int_{0}^{2\pi}|\dot{x}(t)|^{2}dt\geq \mu_{\theta}^{2}\int_{0}^{2\pi}|x(t)-x(t+\theta)|^{2}dt,$$
where $\theta\in(0,2\pi)$ and $\mu_{\theta}=(2\sin\frac{\theta}{2})^{-1}$;
 the equality holds if and only if $x(t)=a \cos t+ b\sin t, a, b\in\mathbb{R}^{d}$.
\end{lemma}
\begin{proof}
Consider the Fourier representation of $x(t)$,  
$$x(t)=\sum_{k\in \mathbb{Z}}c_{k}e^{Jkt},$$ 
where $J=\sqrt{-1}$ and $c_{k}\in\mathbb{C}^{d}$.
Then  $x(t)\in H_{2\pi}^{1}(\mathbb{R},\mathbb{R}^{d})$ and $\int_{0}^{2\pi}x(t)dt=0$ imply that $c_{k}=\bar{c}_{-k}$ and $c_{0}=0$.
By the orthogonality of the basis, we have 
\begin{align*}
 &\ \ \ \mu_{\theta}^{2} \int_{0}^{2\pi}|x(t)-x(t+\theta)|^{2}dt 
=\mu_{\theta}^{2} \int_{0}^{2\pi}|\sum_{k\in\mathbb{Z}}c_{k}e^{Jkt}(1-e^{Jk\theta})|^{2}dt \\
&=\mu_{\theta}^{2} \cdot 2\pi\sum_{k\in\mathbb{Z}}|c_{k}|^{2}|1-e^{Jk\theta}|^{2} 
= 2\pi\sum_{k\in\mathbb{Z}}|c_{k}|^{2}\frac{|1-e^{Jk\theta}|^{2}}{|1-e^{J\theta}|^{2}}\\
&= 2\pi\sum_{k\in\mathbb{Z}}|c_{k}|^{2}\frac{|1-e^{J|k|\theta}|^{2}}{|1-e^{J\theta}|^{2}}
= 2\pi\sum_{k\in\mathbb{Z}}|c_{k}|^{2} |1+e^{J\theta}+e^{J2\theta}+\dots+e^{J(|k|-1)\theta}|^{2}\\
&\leq 2\pi\sum_{k\in\mathbb{Z}}|c_{k}|^{2}k^{2},
\end{align*}
 the equality holds if and only if $k\in\{ 1,-1\}.$ 
 We also notice that $\dot{x}(t)=\sum_{k\in \mathbb{Z}}c_{k}Jke^{Jkt}$, then
 $$\int_{0}^{2\pi}|\dot{x}(t)|^{2}dt=2\pi\sum_{k\in \mathbb{Z}}|c_{k}|^{2}k^{2},$$
 which finish the proof.
\end{proof}
\begin{remark}
Poincar\'{e}-Wirtinger inequality can be seen as an average result of our lemma. Because our lemma implies
$$4 \sin^{2}\frac{\theta}{2}\int_{0}^{2\pi}|\dot{x}(t)|^{2}dt\geq \int_{0}^{2\pi}|x(t)-x(t+\theta)|^{2}dt, \forall \theta\in [0,2\pi],$$
and integrating by $\theta$ from $0$ to $2\pi$, we have
$$4\pi \int_{0}^{2\pi}|\dot{x}(t)|^{2}dt\geq \int_{0}^{2\pi}d\theta\int_{0}^{2\pi}|x(t)-x(t+\theta)|^{2}dt, \forall \theta\in [0,2\pi].$$
Exchanging the  order of integration and using the fact that $\int_{0}^{2\pi}x(t+\theta)d\theta=0$, we can get Poincar\'{e}-Wirtinger inequality.
\end{remark}
Let $\theta_{j}=j\frac{2\pi}{n}$ and $\mu_{j}=\mu_{\theta_{j}}=(2\sin\frac{j\pi}{n})^{-1}$, we have the following corollary.
\begin{corollary}\label{cor3.2}
 For every $x(t)\in \tilde{\Lambda}$, we have
\begin{align*}
\int_0^{2\pi}|\dot{x}(t)|^2dt\geq \sum_{j=1}^{n-1} \nu_j \int_{0}^{2\pi}|x(t)-x(t+j\frac{2\pi}{n})|^{2}dt,
\end{align*}
where 
$\nu_j=\frac{\mu_{j}'\mu_{j}^{2}}{\sum_{j=1}^{n-1}\mu'_{j}}$ and $\mu'_{j}>0$,
 the equality holds if and only if  
$
x(t)=a\cos t+b\sin t,~a,b \in \mathbb{R}^3.
$
\end{corollary}
If we take $\mu'_{j}=(\mu_{j})^{\alpha}$, this corollary is equivalent to the Corollary 2 in \cite{4} which took several pages but played a very important role in their proof.
Another inequality we need is Jensen's inequality,
$$f(\frac{1}{2\pi}\int_0^{2\pi} g(t)dt)\leq \frac{1}{2\pi}\int_0^{2\pi}f(g(t))dt,$$ 
where $f$ is convex. Applying
$f(z)=z^{-\frac{\alpha}{2}}$ and $g(t)=|x(t)-x(t+j\frac{2\pi}{n})|^2$, 
we have
\begin{align*}
[\frac{1}{2\pi}\int_0^{2\pi}|x(t)-x(t+j\frac{2\pi}{n})|^2dt]^{-\frac{\alpha}{2}}\leq \frac{1}{2\pi}\int_0^{2\pi}\frac{1}{|x(t)-x(t+j\frac{2\pi}{n})|^{\alpha}}dt,
\end{align*}
then
\begin{align}\label{2.1}
\int_0^{2\pi}\frac{1}{|x(t)-x(t+j\frac{2\pi}{n})|^{\alpha}}dt\geq (2\pi)^{1+\frac{\alpha}{2}}[\int_0^{2\pi}|x(t)-x(t+j\frac{2\pi}{n})|^2dt]^{-\frac{\alpha}{2}},
\end{align}
where the equality holds if and only if $|x(t)-x(t+j\frac{2\pi}{n})|^{2}\equiv const$.
Moreover, we notice that if the equalities hold  simultaneously for Poincar$\acute{e}$-Wirtinger inequality and Jensen's inequality, we have
 %

\begin{lemma}\label{lemma3.3}
Suppose $ x(t)={a}\cos t+{b}\sin t$, $ {a},{b} \in \mathbb{R}^3$, if for some $\theta\in(0,2\pi)$, $|x(t)-x(t+\theta)|^2,\forall t\in \mathbb{R}$ is constant, we have $ {a}\cdot{b}=0$ and $|{a}|=|{b}|$, which means that $x(t)$ is a uniform circular motion in the plane spanned by $a$ and $b$.
\end{lemma}
\begin{proof}
This lemma is similar to Proposition 3 in \cite{4}. Since
\begin{align*}
&\ \ \ \ |x(t)-x(t+\theta)|^2=[x(t)-x(t+\theta)]\cdot [x(t)-x(t+\theta)]\\
&=|a|^{2}[\cos t-\cos (t+\theta)]^{2}+|b|^{2}[\sin t-\sin (t+\theta)]^{2}+2a\cdot b[\cos t-\cos (t+\theta)] [\sin t-\sin (t+\theta)]\\
&=4\sin^{2}\frac{\theta}{2}[|a|^{2}\sin^{2}(t+\frac{\theta}{2})+|b|^{2}\cos^{2}(t+\frac{\theta}{2})-\sin(2t+\theta) a\cdot b]
\end{align*}
is constant, differentiating by $t$, we have
$$ 4\sin^{2}\frac{\theta}{2} [(|a|^{2}-|b|^{2})\sin(2t+\theta)-2\cos (2t+\theta)a\cdot b]\equiv 0, \forall t\in \mathbb{R}.$$
That is to say $|a|^{2}-|b|^{2}=0$ and $a\cdot b=0$, which finish the proof.
\end{proof}
\section{The proof of Theorem \ref{maintheorem} }\label{section3}
In this section, we give the proof of Theorem  \ref{maintheorem} by   seeking the minimizer of $\widetilde{\mathcal{A}}(x)=\widetilde{\mathcal{A}}_1(x)+\widetilde{\mathcal{A}}_2(x)$ on $\widetilde{\Lambda}$,
where
$$\widetilde{\mathcal{A}}_1(x)=\int_0^{2\pi}[\frac{1+\lambda}{4}|\dot{x}(t)|^2+\frac{M}{|x(t)-C_1|^{\beta}}+\frac{M}{|x(t)-C_2|^{\beta}}]dt,$$
\begin{equation*}
\widetilde{\mathcal{A}}_2(x)=\int_0^{2\pi}[\frac{(1-\lambda)}{4}|\dot{x}(t)|^2dt+\frac{1}{2}\sum_{j=1}^{n-1}\frac{m}{|x(t)-x(t+j\frac{2\pi}{n})|^{\alpha}}]dt,
\end{equation*}
with the parameter $ \lambda=\lambda(m)\in[-1,1]$ to be determinate later.
The idea is, if $\widetilde{\mathcal{A}}_1(x)$ and $\widetilde{\mathcal{A}}_2(x)$ attain their absolute minimum on the same motion, this motion  will also be the minimum of $\widetilde{\mathcal{A}}(x)$ on $\widetilde{\Lambda}$.
 Inspired by the work of Long-Zhang \cite{21}, applying Poincar$\acute{e}$-Wirtinger inequality and Jensen's inequality, we have
\begin{align}
\widetilde{\mathcal{A}}_1(x)&=\int_0^{2\pi}[\frac{1+\lambda}{4}|\dot{x}(t)|^2+\frac{M}{|x(t)-C_1|^{\beta}}+\frac{M}{|x(t)-C_2|\beta}]dt,\nonumber\\
             \label{3.1}               &\geq\int_0^{2\pi}\frac{1+\lambda}{4}|x(t)|^2dt+\int_0^{2\pi}\frac{M}{|x(t)-C_1|^{\beta}}+\frac{M}{|x(t)-C_2|^{\beta}}dt\\
              \label{3.2}                &\geq\int_0^{2\pi}\frac{1+\lambda}{4}|x(t)|^2dt+
              M(2\pi)^{\frac{\beta}{2}+1}[\int_0^{2\pi}|x(t)-C_{1}|^{2}dt]^{-\frac{\beta}{2}}\\
              &\  \  \  \  \  \  +M(2\pi)^{\frac{\beta}{2}+1}[\int_0^{2\pi}|x(t)-C_{2}|^{2}dt]^{-\frac{\beta}{2}} \nonumber \\
                      &= \frac{1+\lambda}{4}\int_0^{2\pi}(|x(t)|^2+1)dt+2M(2\pi)^{\frac{\beta}{2}+1}[\int_0^{2\pi}(|x(t)|^2+1)dt]^{-\frac{\beta}{2}}-\frac{1+\lambda}{2}\pi \nonumber \\
                            &=\Psi(s)
                            \geq \min_{s\geq \sqrt{2\pi}}\Psi(s)\nonumber,
\end{align}
where
$\Psi(s)=\frac{1+\lambda}{4}s^2+2M(2\pi)^{\frac{\beta}{2}+1}s^{-\beta}-\frac{1+\lambda}{2}\pi$ and 
      $s=[\int_0^{2\pi}(|x(t)|^2+1)dt]^{\frac{1}{2}}\geq \sqrt{2\pi}.$
Since $\Psi''(s)>0$ for $s>0$,  $\Psi(s)$ possesses a unique minimum at $s_0=\sqrt{2\pi}(\frac{4\beta M}{1+\lambda})^{\frac{1}{\beta+2}}$; i.e.,
$$\widetilde{\mathcal{A}}_1(x)\geq\Psi(s_{0}),\ \ \forall x(t)\in \widetilde{\Lambda}.$$
Let $\widetilde{x}_1(t)$ be the minimizer of $\widetilde{\mathcal{A}}_1(x)$ on $\widetilde{\Lambda}$,  we will see that 
$\widetilde{\mathcal{A}}_1(\widetilde{x}_1(t))=\Psi(s_{0})$ if and only if $\widetilde{x}_1(t)$ is a special circular motion
where $\widetilde{x}_1(t)$ makes all the inequalities above in the infimum estimate become equalities.
To be precise, the Poincar$\acute{e}$-Wirtinger Inequality applied in (\ref{3.1}) yields
$\widetilde{x}_1(t)={a}_1\cos (t)+{b}_1\sin(t),$ ${a}_1,{b}_1 \in \mathbb{R}^3$.
Jensen's inequality applied in (\ref{3.2})  implies that $|x(t)-C_{j}|\equiv const, j=1,2$.  
Then $|x(t)-C_{1}|^{2}+|x(t)-C_{1}|^{2}=2|x(t)|^{2}+2\equiv const$ and
$|x(t)-C_{1}|^{2}-|x(t)-C_{1}|^{2}\equiv const$. 
The latter one implies that $x(t)$ is in a plane parallel  to $yoz$-plane, but the condition $\int_{0}^{2\pi}x(t)dt=0$ says it must in $yoz$-plane.
So $\widetilde{x}_1(t)$ is a circular periodic motion in $yoz$-plane with the center at origin,
and $s_0=\sqrt{2\pi}(\frac{4\beta M}{1+\lambda})^{\frac{1}{\beta+2}}=[\int_0^{2\pi}(|\widetilde{x}_1(t)|^2+1)dt]^{\frac{1}{2}}$ implies that the radius of the circle 
$$R_1(\lambda)=|\widetilde{x}_1(t)|=\sqrt{(\frac{4\beta M}{1+\lambda})^{\frac{2}{\beta+2}}-1},~~~~
\lambda\in(-1,4\beta M-1)\cap [-1,1].$$
 
Now we turn to the functional $\widetilde{\mathcal{A}}_2(x)$ which is more difficult due to the complexity of its potential part.
So we use Jensen's inequality and our Lemma \ref{mainlemma} instead of Poincar\'{e}-Wirtinger Inequality to give the estimates.
By Corollary \ref{cor3.2} and (\ref{2.1}), we have
\begin{align}
   \widetilde{\mathcal{A}}_2(x)
   &=\int_0^{2\pi}[\frac{(1-\lambda)}{4}|\dot{x}(t)|^2dt+\frac{1}{2}\sum_{j=1}^{n-1}\frac{m}{|x(t)-x(t+j\frac{2\pi}{n})|^{\frac{\alpha}{2}}}]dt\nonumber\\
\label{3.3}                              &\geq {\frac{1-\lambda}{4}}\sum_{j=1}^{n-1}\nu_{j} \int_{0}^{2\pi}|x(t)-x(t+j\frac{2\pi}{n})|^{2}dt+\frac{1}{2}\sum_{j=1}^{n-1}\int_{0}^{2\pi}\frac{m}{|x(t)-x(t+j\frac{2\pi}{n})|^{\frac{\alpha}{2}}}dt\\
 \label{3.4}                             &\geq  {\frac{1-\lambda}{4}}\sum_{j=1}^{n-1}\nu_{j} \int_{0}^{2\pi}|x(t)-x(t+j\frac{2\pi}{n})|^{2}dt+\frac{1}{2}(2\pi)^{1+\frac{\alpha}{2}}m\sum_{j=1}^{n-1}[\int_0^{2\pi}|x(t)-x(t+j\frac{2\pi}{n})|^2dt]^{-\frac{\alpha}{2}}\\
                            &=\sum_{j=1}^{n-1}\Phi_j(\xi_{j}^{x})
                            \geq \sum_{j=1}^{n-1} \min_{s_{j}>0}\Phi_j(s_{j}) \nonumber,
\end{align}
where
\begin{align*}
\Phi_j(s_{j})&= {\frac{1-\lambda}{4}}\nu_{j} s_{j}+\frac{1}{2}(2\pi)^{1+\frac{\alpha}{2}}ms_{j}^{-\frac{\alpha}{2}}, j=1,2,\cdots,n-1,\\
        \xi_{j}^{x}&=\int_0^{2\pi}|x(t)-x(t+j\frac{2\pi}{n})|^2dt.
\end{align*}
Here $\nu_{j}=\frac{\mu_{j}^{2+\alpha}}{\sum_{j=1}^{n-1}\mu^{\alpha}_{j}}$ is selected carefully such that the following equation (\ref{independentj})  can be independent with $j$.
Let $\widetilde{x}_2(t)$ be the minimizer of $\widetilde{\mathcal{A}}_2(x)$ on $\widetilde{\Lambda}$, we claim that $\widetilde{x}_2(t)$ makes all the inequalities above in the infimum estimate become equalities. Because Corollary  \ref{cor3.2} applied in (\ref{3.3}) yields
$\widetilde{x}_2(t)={a}_2\cos (t)+ {b}_2\sin(t),$ ${a}_2,{b}_2 \in \mathbb{R}^3$.
The Jensen's inequality applied in (\ref{3.4})  implies $|\widetilde{x}_2(t)-\widetilde{x}_2(t+j\frac{2\pi}{n})|\equiv const$. By Lemma \ref{lemma3.3} we conclude that $\widetilde{x}_2(t)$ a circular motion  centered at the origin.
Moreover, we see that $\Phi_j(s)$ attains  its  minimum at
$$\bar{s}_{j}=2\pi[\frac{\alpha m}{\nu_{j}(1-\lambda)}]^{\frac{2}{\alpha+2}}=
2\pi(\frac{4\alpha m}{1-\lambda})^{\frac{2}{\alpha+2}}\sin^{2}\frac{j}{n}\pi(\sum_{j=1}^{n-1}\sin^{-\alpha}\frac{j}{n}\pi)^{\frac{2}{\alpha+2}}.$$
Now we pick the radius of the circular motion $R_2(\lambda)=|\widetilde{x}_2(t)|$ appropriately, such that for every $j=1,2,\dots,n-1$,
 $ \xi_{j}^{x}=\bar{s}_{j}$, i.e.,
$$
8\pi R_{2}^{2}(\lambda)\sin^{2}\frac{j}{n}\pi
=2\pi(\frac{4\alpha m}{1-\lambda})^{\frac{2}{\alpha+2}}\sin^{2}\frac{j}{n}\pi(\sum_{j=1}^{n-1}\sin^{-\alpha}\frac{j}{n}\pi)^{\frac{2}{\alpha+2}}.
$$
Then we get
\begin{equation}\label{independentj}
R_2(\lambda)=2^{\frac{-\alpha}{\alpha+2}}(\frac{\alpha m}{1-\lambda}\sum_{j=1}^{n-1}\sin^{-\alpha}\frac{j}{n} \pi)^{\frac{1}{\alpha+2}},
\end{equation}
which is independent with $j$. That is to say the circular motion centered at the origin  with the radius $R_{2}(\lambda)$ is  the absolutely minimum of $\widetilde{\mathcal{A}}_{2}|_{\tilde{\Lambda}}$.

Since $\widetilde{\mathcal{A}}_2(\widetilde{x}_2(t))$ is invariant for every orthogonal transformation, we can choose $\widetilde{x}_2(t)$  such that it is a circular motion in $yoz$-plane.
Then it is enough to set the parameter $\lambda$ appropriately such that  $|\widetilde{x}_1(t)|$ is equal to $|\widetilde{x}_2(t)|$, i.e. $R_{2}=R_{1}$.  
Consider 
\begin{equation}\label{r2-r1}
F(\lambda,m)=R_{2}-R_{1}=2^{\frac{-\alpha}{\alpha+2}}(\frac{\alpha m}{1-\lambda}\sum_{j=1}^{n-1}\sin^{-\alpha}\frac{j}{n} \pi)^{\frac{1}{\alpha+2}}-
\sqrt{(\frac{4\beta M}{1+\lambda})^{\frac{2}{\beta+2}}-1},\end{equation}
it is obvious that, for given $m,M$, $F$ is strictly monotone increasing about $\lambda$,  $ \lim_{\lambda\rightarrow -1^{+}}F=-\infty$ and 
$$ \left\{ \begin{matrix} 
 \lim_{\lambda\rightarrow 1^{-}}F=+\infty,   & if~ 4\beta m \geq 2\\
F(4\beta m-1,m)>0 & if~ 4\beta m<2.
\end{matrix} \right.
$$
So there exists only one $\lambda=\tilde{\lambda}(m)\in (-1,1)$ such that $R_1=R_2$, 
which implies $|\widetilde{x}_1(t)|=|\widetilde{x}_2(t)|$.
Now we investigate the dependence of $\tilde{\lambda}$ on the parameter $m$, differentiating (\ref{r2-r1}) by $m$, we have
$$F_{m}(\tilde{\lambda}(m))+F_{\lambda}(\tilde{\lambda}(m))\frac{d \tilde{\lambda}(m)}{dm}=0.$$
Since $F_{\lambda},F_{m}>0$, we have $ \frac{d \tilde{\lambda}(m)}{dm}<0$. Then $\frac{d R_{1}}{d\lambda}<0$ implies that the radius of the circle increases as $m$ increases.
This completes the proof of Theorem \ref{maintheorem}.
\begin{remark}
For the limiting case $M=0$, there is no center force and it is just the n-body problem with simple choreography constraint. We set 
$\lambda=-1$,  the proof also works and we can get  the uniform circular motions with radius $R_{2}=2^{-\frac{\alpha+1}{\alpha+2}}({\alpha m}\sum_{j=1}^{n-1}\sin^{-\alpha}\frac{j}{n} \pi)^{\frac{1}{\alpha+2}}$ which coincides the result in the work of  Barutello-Terracini \cite{4}.
\end{remark}

\section{Acknowledgment} The first author was  supported by The Youth Fund of Mianyang Normal University and  Project of Sichuan Education Department, and the second author was supported by  the National Natural Science Foundation of China (11601045) and China Scholarship Council. 
Part of this work was done when the second author was visiting University of Minnesota.  He thanks the School of Mathematics and Professor Richard Moeckel for their hospitality and support.

\end{document}